\documentclass[11pt,reqno,a4paper]{amsart}
\usepackage{amssymb,amscd,amsmath}
\usepackage{pspicture}
\usepackage{graphicx}
\usepackage{mathptmx}

\usepackage[matrix,arrow,curve,frame]{xy}    

\xymatrixcolsep{1.9pc}                          
\xymatrixrowsep{1.9pc}
\newdir{ >}{{}*!/-5pt/\dir{>}}                  

 \calclayout

\makeatletter

\renewcommand{\subsection}{\@startsection{subsection}{1}{0pt}{-3.25ex plus -1ex minus-.2ex}{1.5ex plus.2ex}{\normalfont\it}}
\renewcommand{\section}{\@startsection{section}{1}{\parindent}{3.5ex plus 1ex minus .2ex}{2.3ex plus.2ex}{\sc}}

\renewcommand{\phi}{\varphi}

\renewcommand{\geq}{\geqslant}
\renewcommand{\epsilon}{\varepsilon}

\renewcommand{\kappa}{\varkappa}
 
\DeclareMathOperator{\spec}{Spec}

\DeclareMathOperator{\corr}{Cor}

 \DeclareMathOperator{\nis}{nis}

\newcommand{\cc}{\mathcal}
\newcommand{\bb}{\mathbb}

\newtheorem{thm}{Theorem}[section]

\newtheorem{cor}[thm]{Corollary}
\newtheorem{lem}[thm]{Lemma}

\newtheorem{defs}[thm]{Definition}

\makeatother

\begin{document}

\footskip30pt


\title{Comparing motives of smooth algebraic varieties}
\author{Grigory Garkusha}
\address{Department of Mathematics, Swansea University, Fabian Way, Swansea SA1 8EN, UK}
\email{g.garkusha@swansea.ac.uk}


\keywords{Motivic homotopy theory, generalized correspondences, motivic cohomology, triangulated categories of motives}

\subjclass[2010]{14F42; 14F05}

\begin{abstract}
Given a perfect field of exponential characteristic $e$,  the
$\corr$-, $K_0^\oplus$-, $K_0$- and $\bb K_0$-motives of smooth algebraic varieties with
$\bb Z[1/e]$-coefficients are shown to be locally quasi-isomorphic to each other. Moreover, 
it is proved that their triangulated categories
of motives with $\bb Z[1/e]$-coefficients are equivalent. An application
is given for the bivariant motivic spectral sequence.
\end{abstract}

\maketitle

\thispagestyle{empty} \pagestyle{plain}

\newdir{ >}{{}*!/-6pt/@{>}} 


\section{Introduction}

The purpose of the paper is to compare motives of smooth algebraic varieties
corresponding to various categories of correspondences. We also
investigate relations between associated triangulated categories of motives. We work in the framework of
symmetric monoidal strict $V$-categories of correspondences defined in~\cite{GG}.
They are just an abstraction of basic properties of the category of finite correspondences $\corr$.

Given a functor $f:\cc A\to\cc B$ between two such categories of correspondences,
we prove in Theorem~\ref{him} that whenever the base field $k$ is perfect
of exponential characteristic $e$ and $f$ is such that the induced morphisms of complexes of Nisnevich sheaves
   $$f_*:\bb Z_{\cc A}(q)[1/e]\to\bb Z_{\cc B}(q)[1/e],\quad q\geq 0,$$
are quasi-isomorphisms, then for every
$k$-smooth algebraic variety $X$ the morphisms
of twisted motives of $X$ with $\bb Z[1/e]$-coefficients
   $$M_{\cc A}(X)(q)\otimes\bb Z[1/e]\to M_{\cc B}(X)(q)\otimes\bb Z[1/e]$$
are quasi-isomorphisms. Furthermore,
we prove in Theorem~\ref{him} that the induced functors
between triangulated categories of motives
   $$DM_{\cc A}^{eff}(k)[1/e]\to DM_{\cc B}^{eff}(k)[1/e],\quad DM_{\cc A}(k)[1/e]\to DM_{\cc B}(k)[1/e]$$
are equivalences.

Using Theorem~\ref{him} together with theorems of Suslin~\cite{Sus} and Walker~\cite{Wlk} 
(as well as a result of~\cite{GP2}) comparing motivic complexes 
associated to the categories of correspondences
$\corr$, $K_0^\oplus$, $K_0$ and $\bb K_0$, we identify
their motives of smooth algebraic varieties with $\bb Z[1/e]$-coefficients.
Moreover, their triangulated categories
of motives with $\bb Z[1/e]$-coefficients are shown to be equivalent (see Theorem~\ref{him1}).

Another application is given in Theorem~\ref{him2} for the bivariant motivic
spectral sequence in the sense of~\cite{GP}.

Throughout the paper we denote by $Sm/k$ the category of smooth
separated schemes of finite type over the base field $k$.

\section{Preliminaries}

Throughout this paper we work with symmetric monoidal strict $V$-categories
of correspondences in the sense of~\cite{GG}.
The categories $\corr$, $K_0^{\oplus}$, $K_0$, $\bb K_0$ are examples of
such categories (see~\cite{GP1,GP2,Sus,SV1,Wlk} for more details).

Given a symmetric monoidal strict $V$-category of correspondences $\cc A$,
it is standard to define the category of $\cc A$-motives $DM_{\cc A}^{eff}(k)$. By
definition (see~\cite{GG}), it is a full subcategory of the derived category
of Nisnevich sheaves with $\cc A$-correspondences consisting of those complexes
whose cohomology sheaves are $\bb A^1$-invariant. As usual,
the stabilization of $DM_{\cc A}^{eff}(k)$ in the $\bb G_m$-direction
leads to the category $DM_{\cc A}(k)$ (see~\cite{GG} for details).
If $R=\bb Z[S^{-1}]$ is the ring of fractions of $\bb Z$ with respect to a multiplicatively closed
set of intergers $S$, then $\cc A\otimes R$, whose
objects are those of $\cc A$ but morphisms are tensored with $R$, is
a symmetric monoidal strict $V$-category of correspondences (see~\cite{GG}).

\begin{defs}{\rm
Following~\cite{Voe1,SV1,Sus} the {\it $\cc A$-motive of a smooth algebraic variety\/} $X\in Sm/k$, denoted by
$M_{\cc A}(X)$, is the normalized complex of Nisnevich $\cc A$-sheaves associated with the simplicial sheaf
   $$n\longmapsto\cc A(-\times\Delta^n,X)_{\nis},\quad\Delta^n=\spec k[t_0,\ldots,t_n]/(t_0+\cdots+t_n-1).$$

}\end{defs}

In what follows we identify simplicial (pre-)sheaves with their normalized complexes
by using the Dold--Kan correspondence.
Also, if necessary we associate Eilenberg--Mac~Lane $S^1$-spectra to (pre-)sheaves of simplicial Abelian groups.
The reader will always be able to recover any of these associations/identifications.

We have a motivic bispectrum
   \begin{equation*}\label{bispa}
    M_{\cc A}^{\bb G_m}(X):=(EM(M_{\cc A}(X)),EM(M_{\cc A}(X\wedge\bb G_m^{\wedge 1})),
        EM(M_{\cc A}(X\wedge\bb G_m^{\wedge2})),\ldots),
   \end{equation*}
where each entry $EM(M_{\cc A}(X\wedge\bb G_m^{\wedge q})$, $q\geq 0$, is the Eilenberg--Mac~Lane 
$S^1$-spectrum associated with the simplicial $\cc A$-sheaf
$n\longmapsto\cc A(-\times\Delta^n,X\wedge\bb G_m^{\wedge q})_{\nis}$.
$M_{\cc A}(X\wedge\bb G_m^{\wedge q})$ will also be denoted by $M_{\cc A}(X)(q)$. In what follows we denote by
$\bb Z_{\cc A}(q)$ the complex $M_{\cc A}(pt)(q)[-q]$, $pt:=\spec k$ (the shift is cohomological).

\begin{defs}\label{zap}{\rm
Following terminology of~\cite[Section~6]{GP} the {\it bivariant $\cc A$-motivic cohomology groups\/} are defined by
    $$H_{\cc A}^{p,q}(X,Y):=H^p_{\nis}(X,\cc A(-\times\Delta^\bullet,Y\wedge\bb G_m^{\wedge q})_{\nis}[-q]),$$
where the right hand side stands for Nisnevich hypercohomology groups of $X$ with coeffitients in
$\cc A(-\times\Delta^\bullet,Y\wedge\bb G_m^{\wedge n})_{\nis}[-q]$ (the shift is cohomological). If
$\cc A=\corr$, we shall write $H_{\cc M}^{p,q}(X,Y)$ to denote $H_{\cc A}^{p,q}(X,Y)$. We also call
$H_{\cc M}^{*,*}(X,Y)$ the {\it bivariant motivic cohomology groups}.

Following~\cite{GP2} we say that the bigraded presheaves $H^{*,*}_{\cc A}(-,Y)$ satisfy
the {\it cancellation property\/} if all maps
   $$\beta^{p,q}:H^{p,q}_{\cc A}(X,Y)\to H^{p+1,q+1}_{\cc A}(X\wedge\bb G_m,Y)$$
induced by the structure maps of the spectrum $M_{\cc A}^{\bb G_m}(Y)$ are
isomorphisms.

}\end{defs}

Given $Y\in Sm/k$, denote by
   $$M_{\cc A}^{\bb G_m}(Y)_f:=(EM(M_{\cc A}(Y))_f,EM(M_{\cc A}(Y\wedge\bb G_m^{\wedge 1}))_f,EM(M_{\cc A}(Y\wedge\bb G_m^{\wedge2}))_f,\ldots),$$
where each $EM(M_{\cc A}(Y\wedge\bb G_m^{\wedge n}))_f$ is a fibrant replacement of $EM(M_{\cc A}(Y\wedge\bb G_m^{\wedge n}))$
in the injective local stable model structure of motivic $S^1$-spectra.
It is worth to note that each $EM(M_{\cc A}(Y\wedge\bb G_m^{\wedge n}))_f$ can be constructed within
the category of chain complexes of Nisnevich $\cc A$-sheaves and then taking the corresponding $S^1$-spectrum
(this can be shown similar to~\cite[5.12]{GP}).

\begin{lem}[see \cite{GG}]\label{brus}
The bigraded presheaves $H^{*,*}_{\cc A}(-,Y)$ satisfy the
cancellation property if and only if $M_{\cc A}^{\bb G_m}(Y)_f$ is 
motivically fibrant as an ordinary motivic bispectrum.
\end{lem}

In what follows we shall write $SH(k)$ to denote the stable homotopy category of motivic bispectra.

\begin{cor}[see \cite{GG}]\label{bruscor}
The presheaves $H^{*,*}_{\cc A}(-,Y)$ are represented in $SH(k)$ by the
bispectrum $M_{\cc A}^{\bb G_m}(Y)_f$. Precisely,
   $$H^{p,q}_{\cc A}(X,Y)=SH(k)(X_+,S^{p,q}\wedge M_{\cc A}^{\bb G_m}(Y)_f),\quad p,q\in\bb Z,$$
where $S^{p,q}=S^{p-q}\wedge\bb G_m^{\wedge q}$.
\end{cor}

\section{Motivic complexes and triangulated categories of motives}

Throughout this section we assume $f:\cc A\to\cc B$ to be a functor of symmetric
monoidal strict $V$-categories of correspondences satisfying the cancellation property.
We always assume that $f$ is the identity map on objects. 


\begin{thm}\label{him}
Suppose $k$ is a perfect field of exponential characteristic $e$. 
If the morphism of complexes of Nisnevich sheaves
   $$f_*:\bb Z_{\cc A}(q)[1/e]=\bb Z_{\cc A\otimes\bb Z[1/e]}(q)\to
        \bb Z_{\cc B}(q)[1/e]=\bb Z_{\cc B\otimes\bb Z[1/e]}(q),\quad q\geq 0,$$
is a quasi-isomorphism, then for every $X\in Sm/k$ the morphism
of twisted motives of $X$ with $\bb Z[1/e]$-coefficients
   $$M_{\cc A}(X)(q)\otimes\bb Z[1/e]\to M_{\cc B}(X)(q)\otimes\bb Z[1/e]$$
is a quasi-isomorphism. Furthermore, the induced functor
   $$DM_{\cc A}^{eff}(k)[1/e]\to DM_{\cc B}^{eff}(k)[1/e]$$
is an equivalence of triangulated categories.
\end{thm}

\begin{proof}
By Lemma~\ref{brus} the bispectra $M_{\cc A\otimes\bb Z[1/e]}^{\bb G_m}(X)_f$ and $M_{\cc B\otimes\bb Z[1/e]}^{\bb G_m}(X)_f$
are motivically fibrant. Note that $M_{\cc A\otimes\bb Z[1/e]}^{\bb G_m}(X)_f$ and $M_{\cc B\otimes\bb Z[1/e]}^{\bb G_m}(|X)_f$ are
fibrant replacements of the bispectra
   $$H(\cc A\otimes\bb Z[1/e])(X):=(EM(\cc A(-,X)\otimes\bb Z[1/e]),EM(\cc A(-,X\wedge\bb G_m^{\wedge 1})\otimes\bb Z[1/e]),\ldots)$$
and
   $$H(\cc B\otimes\bb Z[1/e])(X):=(EM(\cc B(-,X)\otimes\bb Z[1/e]),EM(\cc B(-,X\wedge\bb G_m^{\wedge 1})\otimes\bb Z[1/e]),\ldots)$$
respectively, where ``$EM$" stands for the Eilenberg--Mac~Lane (symmetric) $S^1$-spectrum.

By our assumption, the natural morphism of bispectra
   \begin{equation}\label{chtoto}
    M_{\cc A\otimes\bb Z[1/e]}^{\bb G_m}(pt)_f\to M_{\cc B\otimes\bb Z[1/e]}^{\bb G_m}(pt)_f
   \end{equation}
induced by $f$ is a level local Nisnevich equivalence. Hence it is a level schemewise equivalence, because both bispectra
are motivically fibrant. Since $H(\cc A\otimes\bb Z[1/e])(X)\to M_{\cc A\otimes\bb Z[1/e]}^{\bb G_m}(X)_f$
and $H(\cc B\otimes\bb Z[1/e])(X)\to M_{\cc B\otimes\bb Z[1/e]}^{\bb G_m}(X)_f$ are level motivic equivalences,
then $H(\cc A\otimes\bb Z[1/e])(pt)\to H(\cc B\otimes\bb Z[1/e])(pt)$ is a level motivic equivalence of bispectra.

Consider a commutative diagram of bispectra
   $$\xymatrix{H(\cc A\otimes\bb Z[1/e])(pt)\wedge X_+\ar[d]\ar[r]&H(\cc A\otimes\bb Z[1/e])(X)\ar[r]\ar[d]&M_{\cc A\otimes\bb Z[1/e]}^{\bb G_m}(X)_f\ar[d]\\
               H(\cc B\otimes\bb Z[1/e])(pt)\wedge X_+\ar[r]&H(\cc B\otimes\bb Z[1/e])(X)\ar[r]&M_{\cc B\otimes\bb Z[1/e]}^{\bb G_m}(X)_f}$$
Since $H(\cc A\otimes\bb Z[1/e])(pt)\to H(\cc B\otimes\bb Z[1/e])(pt)$ is a level motivic equivalence of bispectra,
then so is the left vertical arrow of the diagram by~\cite[ 12.7]{Jar07}.
By the proof of generalized R\"ondigs--{\O}stv{\ae}r's theorem~\cite[5.3]{GG}
the left horizontal arrows are stable motivic equivalences, and hence so is the middle vertical
map. Since the right horizontal arrows are stable motivic equivalences, then so is the right vertical map.
But it is a stable motivic equivalence between motivically fibrant bispectra, and so it is a level
schemewise equivalence.

We see that each morphism of $S^1$-spectra
   $$EM(M_{\cc A\otimes\bb Z[1/e]}(X)(q))_f\to EM(M_{\cc B\otimes\bb Z[1/e]}(X)(q))_f$$
is a schemewise stable equivalence. But every such arrow is a local replacement of the morphism
$EM(M_{\cc A\otimes\bb Z[1/e]}(X)(q))\to EM(M_{\cc B\otimes\bb Z[1/e]}(X)(q))$. It follows that the latter arrow
is a local equivalence, and hence
   $$M_{\cc A}(X)(q)\otimes\bb Z[1/e]\to M_{\cc B}(X)(q)\otimes\bb Z[1/e]$$
is a quasi-isomorphism of complexes of Nisnevich sheaves.

Now to prove that the induced functor
   $$f:DM_{\cc A}^{eff}(k)[1/e]\to DM_{\cc B}^{eff}(k)[1/e]$$
is an equivalence of triangulated categories, we use a standard argument
for compactly generated triangulated categories. Precisely, it suffices to show that
the image of compact generators of the left category is a set of compact
generators of the right category and that Hom-sets between them on the right
and on the left are isomorphic. The families $\{M_{\cc A}(X)\otimes\bb Z[1/e][n]\mid X\in Sm/k,n\in\bb Z\}$,
$\{M_{\cc B}(X)\otimes\bb Z[1/e][n]\mid X\in Sm/k,n\in\bb Z\}$ are sets of compact generators
for $DM_{\cc A}^{eff}(k)[1/e]$ and $DM_{\cc B}^{eff}(k)[1/e]$ respectively.

The functor $f$ maps one family to another by construction. Also, Hom-sets between compact generators
from the first (respectively second) family is given by bivariant $\cc A$-motivic
cohomology $H^{*,*}_{\cc A}(X,Y)\otimes\bb Z[1/e]$ (respectively $H^{*,*}_{\cc B}(X,Y)\otimes\bb Z[1/e]$).
By the first part of the proof $F$ induces a schemewise stable equivalence of bispectra
   $$M_{\cc A\otimes\bb Z[1/e]}^{\bb G_m}(X)_f\to M_{\cc B\otimes\bb Z[1/e]}^{\bb G_m}(X)_f.$$
It follows from Corollary~\ref{bruscor} that the homomorphism
   $$f:H^{*,*}_{\cc A}(X,Y)\otimes\bb Z[1/e]\to H^{*,*}_{\cc B}(X,Y)\otimes\bb Z[1/e]$$
is an isomorphism, as was to be shown.
\end{proof}

The following statement is proved similar to the second part of Theorem~\ref{him}.

\begin{cor}\label{himcor}
Under the assumptions of Theorem~\ref{him} the canonical functor
   $$DM_{\cc A}(k)[1/e]\to DM_{\cc B}(k)[1/e]$$
is an equivalence of triangulated categories.
\end{cor}

The proof of Theorem~\ref{him} also allows to compare motives of
certain smooth algebraic varieties without inverting the
exponential characteristic of the base field. This is possible whenever
$U\in Sm/k$ is dualizable in $SH(k)$. For instance, it is shown in~\cite[Appendix]{Hu}
that any smooth projective variety $U\in Sm/k$ is dualizable in $SH(k)$ over any field $k$.
Namely, the following result is true:

\begin{thm}\label{dualaz}
Suppose $k$ is any field and $U\in Sm/k$ is dualizable in $SH(k)$ (e.g. $U$ is
a smooth projective variety). If the morphism of complexes of Nisnevich sheaves
   $f_*:\bb Z_{\cc A}(q)\to\bb Z_{\cc B}(q)$, $q\geq 0,$
is a quasi-isomorphism, then so is the morphism of twisted motives
   $M_{\cc A}(U)(q)\to M_{\cc B}(U)(q).$
\end{thm}

\begin{proof}
Since $U$ is dualizable in $SH(k)$, the proof of generalized R\"ondigs--{\O}stv{\ae}r's
theorem~\cite[5.3]{GG} shows that $H(\cc A)(pt)\wedge U_+\to H(\cc A)(U)$ and
$H(\cc B)(pt)\wedge U_+\to H(\cc B)(U)$ are stable motivic equivalences. It remains to
repeat the proof of the first part of Theorem~\ref{him} word for word.
\end{proof}

\begin{cor}\label{himcor2}
Suppose $k$ is any field and $\cc F$ is a family of smooth algebraic varieties
which are dualizable in $SH(k)$. Also, suppose that the morphism of 
complexes of Nisnevich sheaves $f_*:\bb Z_{\cc A}(q)\to\bb Z_{\cc B}(q)$, $q\geq 0,$
is a quasi-isomorphism.
Let $DM_{\cc A}^{eff}(k)\langle\cc F\rangle$ and $DM_{\cc B}^{eff}(k)\langle\cc F\rangle$
(respectively $DM_{\cc A}(k)\langle\cc F\rangle$ and $DM_{\cc B}(k)\langle\cc F\rangle$) be full
compactly generated triangulated subcategories of $DM_{\cc A}^{eff}(k)$ and $DM^{eff}_{\cc B}(k)$
(respectively $DM_{\cc A}(k)$ and $DM_{\cc B}(k)$)
generated by the motives $\{M_{\cc A}(U)[n]\mid U\in\cc F,n\in\bb Z\}$ and
$\{M_{\cc B}(U)[n]\mid U\in\cc F,n\in\bb Z\}$ (respectively
$\{M_{\cc A}^{\bb G_m}(U)\otimes\bb G_m^{\wedge q}[n]\mid U\in\cc F,n,q\in\bb Z\}$ and
$\{M_{\cc B}^{\bb G_m}(U)\otimes\bb G_m^{\wedge q}[n]\mid U\in\cc F,n,q\in\bb Z\}$).
Then the canonical functors
   $$DM_{\cc A}^{eff}(k)\langle\cc F\rangle\to DM_{\cc B}^{eff}(k)\langle\cc F\rangle,\quad
     DM_{\cc A}(k)\langle\cc F\rangle\to DM_{\cc B}(k)\langle\cc F\rangle$$
are equivalences of triangulated categories.
\end{cor}

\begin{proof}
This follows from Theorem~\ref{dualaz} if we use the same proof as for the second part of Theorem~\ref{him}.
\end{proof}

Suppose $k$ is a perfect field. Consider natural functors between categories of correspondences
   $$\bb K_0\xrightarrow{\alpha} K_0^\oplus\xrightarrow{\beta} K_0\xrightarrow{\gamma}\corr,$$
where $\bb K_0$ is defined in~\cite{GP1}. All of these categories of correspondences
are symmetric monoidal strict $V$-categories of correspondences satisfying the cancellation
property. Moreover, $\alpha,\beta,\gamma$ are strict symmetric monoidal functors
which are the identities on objects.
They induce morphisms of complexes of Nisnevich sheaves
   $$\bb Z_{\bb K_0}(q)\xrightarrow{\alpha}\bb Z_{K_0^\oplus}(q)\xrightarrow{\beta}\bb Z_{K_0}(q)\xrightarrow{\gamma}\bb Z(q),\quad q\geq 0.$$
By Suslin's theorem~\cite{Sus} $\gamma\beta$ is a quasi-isomorphism. Walker~\cite[6.5]{Wlk1} proved that
$\gamma$ is a quasi-isomorphism. We see that $\beta$ is a quasi-isomorphism as well.
Also, $\alpha$ is an isomorphism by~\cite[7.2]{GP1} (over any base field).

As an application of Theorem~\ref{him} and Corollary~\ref{himcor} we can now deduce the following

\begin{thm}\label{him1}
Suppose $k$ is a perfect field of exponential characteristic $e$.
Then for every $X\in Sm/k$ the morphism
of twisted motives of $X$ with $\bb Z[1/e]$-coefficients
   $$M_{\bb K_0}(X)(q)\otimes\bb Z[1/p]\xrightarrow{\alpha}M_{K_0^\oplus}(X)(q)\otimes\bb Z[1/e]\xrightarrow{\beta}
     M_{K_0}(X)(q)\otimes\bb Z[1/e]\xrightarrow{\gamma}M(X)(q)\otimes\bb Z[1/e]$$
are quasi-isomorphisms of complexes of Nisnevich sheaves. Furthermore, the induced functors
   $$DM_{\bb K_0}^{eff}(k)[1/e]\xrightarrow{\alpha}DM_{K_0^\oplus}^{eff}(k)[1/e]\xrightarrow{\beta}
     DM_{K_0}^{eff}(k)[1/e]\xrightarrow{\gamma}DM^{eff}(k)[1/e]$$
and
   $$DM_{\bb K_0}(k)[1/e]\xrightarrow{\alpha}DM_{K_0^\oplus}(k)[1/e]\xrightarrow{\beta}
     DM_{K_0}(k)[1/e]\xrightarrow{\gamma}DM(k)[1/e]$$
are equivalences of triangulated categories.
\end{thm}

We also have the following application of Theorem~\ref{dualaz}.

\begin{thm}\label{dualaz2}
Suppose $k$ is a perfect field and $U\in Sm/k$ is dualizable in $SH(k)$ (e.g. $U$ is
a smooth projective variety). Then the morphisms of twisted motives
   $$M_{\bb K_0}(U)(q)\xrightarrow{\alpha}M_{K_0^\oplus}(U)(q)\xrightarrow{\beta}
     M_{K_0}(U)(q)\xrightarrow{\gamma}M(U)(q),\quad q\geq 0,$$
are quasi-isomorphisms of complexes of Nisnevich sheaves.
\end{thm}

\begin{cor}\label{himcor3}
Suppose $k$ is a perfect field and $\cc F$ is a family of smooth algebraic varieties
which are dualizable in $SH(k)$. Under the notation of Corollary~\ref{himcor2} the functors
   $$DM_{\bb K_0}^{eff}(k)\langle\cc F\rangle\xrightarrow{\alpha}DM_{K_0^\oplus}^{eff}(k)\langle\cc F\rangle\xrightarrow{\beta}
     DM_{K_0}^{eff}(k)\langle\cc F\rangle\xrightarrow{\gamma}DM^{eff}(k)\langle\cc F\rangle$$
and
   $$DM_{\bb K_0}(k)\langle\cc F\rangle\xrightarrow{\alpha}DM_{K_0^\oplus}(k)\langle\cc F\rangle\xrightarrow{\beta}
     DM_{K_0}(k)\langle\cc F\rangle\xrightarrow{\gamma}DM(k)\langle\cc F\rangle$$
are equivalences of triangulated categories.
\end{cor}

Another application of Theorem~\ref{him} is for the bivariant motivic
spectral sequence in the sense of~\cite[7.9]{GP}
   $$E_2^{pq}=H^{p-q,-q}_{K_0^\oplus}(U,X)\Longrightarrow K_{-p-q}(U,X),\quad U\in Sm/k.$$
It is Grayson's motivic specttral sequence~\cite{Gr} applied to bivariant algebaraic
$K$-theory of smooth algebraic varieties (see~\cite[Section~7]{GP} for details and definitions).
The spectral sequence is strongly convergent and the following relation is true (over perfect fields) by~\cite[7.8]{GP}:
   $$H^{p,q}_{K_0^\oplus}(U,X)=DM_{K_0^\oplus}^{eff}(k)(M_{K_0^\oplus}(U),M_{K_0^\oplus}(X)(q)[p-2q]).$$
The latter relation together with Theorem~\ref{him1} imply the relation
   $$H^{p,q}_{K_0^\oplus}(U,X)\otimes\bb Z[1/e]=H^{p,q}_{\cc M}(U,X)\otimes\bb Z[1/e].$$
In turn, if $X\in Sm/k$ is dualizable in $SH(k)$, then Theorem~\ref{dualaz2} implies the relation
  $$H^{p,q}_{K_0^\oplus}(U,X)=H^{p,q}_{\cc M}(U,X).$$

Thus we have proved the following

\begin{thm}\label{him2}
Suppose $k$ is a perfect field of exponential characteristic $e$.
Then the bivariant motivic spectral sequence with $\bb Z[1/e]$-coefficients takes the form
   $$E_2^{pq}=H^{p-q,-q}_{\cc M}(U,X)\otimes\bb Z[1/e]\Longrightarrow K_{-p-q}(U,X)\otimes\bb Z[1/e]$$
for all $U, X\in Sm/k$. Furthermore, if $X\in Sm/k$ is dualizable in $SH(k)$ (e.g. $X$ is
a smooth projective variety), then for every $U\in Sm/k$ the bivariant motivic spectral 
sequence takes the form
   $$E_2^{pq}=H^{p-q,-q}_{\cc M}(U,X)\Longrightarrow K_{-p-q}(U,X).$$
\end{thm}

Thus the preceding theorem says that the classical motivic spectral sequence starting from
motivic cohomology and converging to $K$-theory can be extended to bivariant motivic cohomology and
bivariant $K$-theory on all smooth algebraic varieties after inverting the exponential characteristic.
In turn, if the second argument is dualizable in $SH(k)$ then the exponential characteristic
inversion is not necessary.

\end{document}